\documentclass{amsart}

\input xy
\xyoption{all}

\usepackage{geometry}
\usepackage{amsmath}
\usepackage{amssymb}
\usepackage{cite}
\usepackage{amsthm}  
\usepackage{tikz}
\usetikzlibrary{er,positioning}

\newtheorem{same}{This should never appear}[section]
\newtheorem{defin}[same]{Definition}

\newtheorem{remark}[same]{Remark}
\newtheorem{theorem}[same]{Theorem}
\newtheorem{example}[same]{Example}
\newtheorem{lemma}[same]{Lemma}
\newtheorem{fact}[same]{Fact}
\newtheorem{question}[same]{Question}
\newtheorem{cor}[same]{Corollary}
\newtheorem{prop}[same]{Proposition}

\newtheorem{hypothesis}[same]{Hypothesis}
\newtheorem{nota}[same]{Notation}

\newtheorem{defin*}{Definition}
\newtheorem*{theorem*}{Theorem}
\newcommand{\bb}{\mathbf{b}}
\newcommand{\rest}{\mathord{\upharpoonright}}
\newcommand{\id}{\textrm{id}}
\newcommand{\eff}{\mathcal{F}}
\newcommand{\tte}{\mathcal{T}}

\newcommand{\K}{\mathbf{K}}

\newcommand{\Ka}{\K^{T}}
\newcommand{\LS}{\operatorname{LS}}

\newcommand{\leap}[1]{\le_{#1}}

\newcommand{\lea}{\leap{\K}}

\newcommand{\gtp}{\mathbf{gtp}}
\newcommand{\gS}{\mathbf{gS}}

\newcommand{\tp}[3]{\mathrm{tp}(#1/#2, #3)}

\newcommand{\tupop}{\textup{(}}
\newcommand{\tupcp}{\textup{)}}

\DeclareMathOperator{\Inv}{Inv}            
\DeclareMathOperator{\pp}{pp}    
\DeclareMathOperator{\cof}{cf}    

\title{On universal modules with pure embeddings}
\date{\today.} 

\author{Thomas G.\ Kucera}
\email{Thomas.Kucera@umanitoba.ca}
\urladdr{https://server.math.umanitoba.ca/~tkucera/}
\address{Department of Mathematics \\ University of Manitoba \\
Winnipeg, Manitoba, CA}

\author{Marcos Mazari-Armida}
\email{mmazaria@andrew.cmu.edu}
\urladdr{http://www.math.cmu.edu/~mmazaria/ }
\address{Department of Mathematical Sciences \\ Carnegie Mellon
University \\ Pittsburgh, Pennsylvania, USA}

\begin{document}

\begin{abstract}
We show that certain classes of modules have universal models with
respect to pure embeddings.

\begin{theorem} Let $R$ be a ring, $T$ a first-order theory with an
infinite model extending the theory of $R$-modules and $\K^T=(Mod(T),
\leq_{pp})$ \tupop where $\leq_{pp}$ stands for pure submodule\tupcp.
Assume
$\K^T$ has joint embedding and amalgamation.

 If $\lambda^{|T|}=\lambda$ or $\forall \mu < \lambda( \mu^{|T|} <
\lambda)$, then $\K^T$ has a universal model of cardinality $\lambda$.
\end{theorem}

As a special case we get a recent result of Shelah \cite[1.2]{sh820}
concerning the existence of universal reduced torsion-free abelian
groups with respect to pure embeddings.

 We begin the study of limit models for classes of $R$-modules with
joint embedding and amalgamation. We show that limit models with chains of long cofinality are pure-injective and we characterize limit models with chains of countable cofinality. This can be used to answer Question
4.25 of \cite{maz}.

As this paper is aimed at model theorists and algebraists an effort
was made to provide the background for both.
\end{abstract}


\maketitle

{\let\thefootnote\relax\footnote{{AMS 2010 Subject Classification:
Primary: 03C48. Secondary: 03C45, 03C60, 13L05, 16D10.
}}}

\tableofcontents

\section{Introduction}

The first result concerning the existence of universal uncountable
objects in classes of modules was \cite{eklof}. In it, Eklof showed
that there exists a homogeneous universal $R$-module of cardinality
$\lambda$ in the class of $R$-modules if and only if $\lambda^{<
\gamma} =\lambda$ (where $\gamma$ is the least cardinal such that
every ideal of $R$ is generated by less than $\gamma$ elements).

Grossberg and Shelah \cite{grsh} used the
weak continuum hypothesis to answer a question of Macintyre and
Shelah \cite{mash} regarding the existence of universal locally
finite groups in uncountable cardinalities. Kojman and Shelah
\cite{kojsh} and Shelah \cite{sh1}, \cite{sh2}, \cite{sh3} and
\cite{sh820} continued the study of universal groups for certain
classes of abelian groups with respect to embeddings and pure
embeddings. For further historical comments the reader can consult
\cite[\S 6]{dz}.  

In this paper, we will give a positive answer to the question of
whether certain classes of modules with pure embeddings have
universal models in specific cardinals. More precisely, we obtain:

\textbf{Theorem \ref{mainc}.}
\textit{Let $R$ be a ring, $T$ a first-order theory with an infinite
model extending the theory of $R$-modules and $\K^T=(Mod(T),
\leq_{pp})$ \tupop where $\leq_{pp}$ stands for 
pure submodule\tupcp. Assume
$\K^T$ has joint embedding and 
amalgamation.}

 \textit{If $\lambda^{|T|}=\lambda$ or $\forall \mu < \lambda(
\mu^{|T|} < \lambda)$, then $\K^T$ has a universal model of
cardinality
$\lambda$.}

There are many  examples of theories satisfying the hypothesis of
Theorem \ref{mainc} (see Example \ref{ex}). One of them is the theory
of torsion-free abelian groups. So as
straightforward corollary we get:

\textbf{Corollary \ref{u-tf}.}
\textit{If $\lambda^{\aleph_0}=\lambda$ or $\forall \mu < \lambda(
\mu^{\aleph_0} < \lambda)$, then the class of torsion-free abelian
groups with pure embeddings has a universal group of cardinality
$\lambda$.}

In \cite[1.2]{sh820} Shelah shows a result analogous to the above
theorem, but instead of working with the class of torsion-free
abelian groups he works with the class of reduced torsion-free
abelian groups. The reason Corollary \ref{u-tf} transfers to Shelah's
setting is because  every abelian group can be written as a direct sum of a
unique divisible subgroup and a unique up to isomorphism reduced
subgroup (see \cite[\S 4.2.5]{fuchs}). Shelah's statement is
Corollary \ref{she-res} in this paper.

The proof presented here is not a generalization of Shelah's original
idea. We prove first that the class is $\lambda$-Galois-stable (for $\lambda^{|T|}=\lambda$) and then using
that the class is an abstract elementary class we construct
universal extensions of size $\lambda$  (for
$\lambda^{|T|}=\lambda$). By contrast, Shelah first constructs
universal
extensions of cardinality $\lambda$ (for
$\lambda^{\aleph_0}=\lambda$) and from it he concludes that the class
is $\lambda$-Galois-stable.

The methods used in both proofs are also quite different. We
exploit the fact that any theory of $R$-modules has
$pp$-quantifier elimination and that our class is an abstract elementary class with joint embedding and amalgamation. By contrast,
Shelah's argument seems to
only work in the restricted setting of torsion-free abelian groups. This is
the case since the main device of  his argument is the existence of a
metric in reduced torsion-free abelian groups and the completions
obtained from this metric.

In \cite{maz}, the second author began the study of limit models in
classes of abelian groups. In this paper we go one step further and
begin the study of  limit models in classes of $R$-modules with joint
embedding and amalgamation.  \emph{Limit models} were introduced in
\cite{kosh} as a substitute for saturation in the context of AECs.
 Intuitively the reader can think of them as universal models with some level of homogeneity (see Definition \ref{limit}).
They have proven to be an important concept in tackling Shelah's
eventual categoricity conjecture. The key question has been the
uniqueness of limit models of the same cardinality but of different length.\footnote{A more
detailed account of the importance of limit models is given in
\cite[\S 1]{maz}.}

We show that limit models in $\K^T$ are elementary equivalent (see Lemma \ref{limeq}).
We generalize \cite[4.10]{maz} by showing that limit models with
chains of cofinality greater than $ |T|$ are pure-injective (see
Theorem \ref{bigpi}). We characterize limit models with chains of
countable cofinality for classes that are closed under direct sums
(see Theorem \ref{countablelim}). The main feature is that there is a
natural way to construct universal models over pure-injective
modules. More precisely, given $M$ pure-injective and $U$ a universal
model of size $\| M \|$, $M \oplus U$ is universal over $M$. As a by-product of our study of limit models and \cite[4.15]{maz}
we answer Question 4.25 of \cite{maz}.

\textbf{Theorem \ref{main1}.}
\textit{If $G$ is a $(\lambda, \omega)$-limit model in the class of
torsion-free abelian groups with pure embeddings, then $G \cong
\mathbb{Q}^{(\lambda)} \oplus \prod_{p}
\overline{\mathbb{Z}_{(p)}^{(\lambda)}}^{(\aleph_0)}$.}

Finally, combining Corollary \ref{u-tf} and Theorem \ref{main1}, we
are able to construct universal extensions of cardinality $\lambda$
for some cardinals such that the class of torsion-free groups with
pure embeddings is not $\lambda$-Galois-stable  (an example for such
a $\lambda$ is $\beth_\omega$). This is the first example of an AEC
with joint embedding, amalgamation and no maximal models in which one
can construct universal extensions of cardinality $\lambda$ without
the hypothesis of $\lambda$-Galois-stability.

The paper is organized as follows. Section 2 presents necessary
background. Section 3 studies classes of the form $\Ka$, studies universal models in these classes
and shows how \cite[1.2]{sh820} is a special case of the theory
developed in the section. Section 4 begins the study of limit models
for classes of $R$-modules with joint embedding and 
amalgamation.  It also
answers Question 4.25 of \cite{maz}.

This paper was written while the second author was working on a Ph.D.
under the direction of Rami Grossberg at Carnegie Mellon University
and he would like to thank Professor Grossberg for his guidance and
assistance in his research in general and in this work in particular. We would also like to thank Sebastien Vasey for several comments that helped improve the paper. We would also like to thank John T. Baldwin for introducing us to one another and for useful comments that improved the paper. We are grateful to the referees for their comments that significantly improved the paper.

\section{Preliminaries}

We introduce the key concepts of abstract elementary classes and the
model theory of modules that are used in this paper. Our primary 
references for the former are
 \cite[\S 4 - 8]{baldwinbook09} and  \cite[\S 2, \S
4.4]{ramibook}. Our primary references for the latter is
\cite{prest}.

\subsection{Abstract elementary classes}

Abstract elementary classes (AECs) were introduced by Shelah in
\cite{sh88}. Among the requirements we have that an AEC is closed under directed colimits and that every set
is contained in a small model in the class. Given a model $M$, we
will write $|M|$ for its underlying set and $\| M \|$ for its
cardinality.

\begin{defin}\label{aec-def}
  An \emph{abstract elementary class} is a pair $\K = (K, \lea)$,
where:

  \begin{enumerate}
    \item $K$ is a class of $\tau$-structures, for some fixed
language $\tau = \tau (\K)$. 
    \item $\lea$ is a partial order on $K$. 
    \item $(K, \lea)$ respects isomorphisms: 
    
    If $M \lea N$ are in $K$
and $f: N \cong N'$, then $f[M] \lea N'$. 

In particular 
\tupop taking $M =
N$\tupcp, $K$ is closed under isomorphisms.
    \item If $M \lea N$, then $M \subseteq N$. 
    \item Coherence: If $M_0, M_1, M_2 \in K$ satisfy $M_0 \lea M_2$,
$M_1 \lea M_2$, and $|M_0| \subseteq |M_1|$, then $M_0 \lea M_1$.
    \item Tarski-Vaught axioms: Suppose $\delta$ is a limit ordinal
and $\{ M_i \in K : i < \delta \}$ is an increasing chain. Then:

        \begin{enumerate}

            \item $M_\delta := \bigcup_{i < \delta} M_i \in K$ and
$M_i \lea M_\delta$ for every $i < \delta$.
            \item\label{smoothness-axiom}Smoothness: If there is some
$N \in K$ so that for all $i < \delta$ we have $M_i \lea N$, then we
also have $M_\delta \lea N$.

        \end{enumerate}

    \item L\"{o}wenheim-Skolem-Tarski axiom: There exists a cardinal
$\lambda \ge |\tau(\K)| + \aleph_0$ such that for any $M \in K$ and
$A \subseteq |M|$, there is some $M_0 \lea M$ such that $A \subseteq
|M_0|$ and $\|M_0\| \le |A| + \lambda$. We write $\LS (\K)$ for the
minimal such cardinal.
  \end{enumerate}
\end{defin}

\begin{nota}\
\begin{itemize}
\item If $\lambda$ is cardinal and $\K$ is an AEC, then $\K_{\lambda}=\{ M \in \K : \| M \|=\lambda \}$.

\item Let $M, N \in \K$. If we write ``$f: M \to N$" we assume that
$f$ is a $\K$-embedding, i.e., $f: M \cong f[M]$ and $f[M] \lea N$.
In particular $\K$-embeddings are always monomorphisms.

\item  Let $M, N \in \K$ and $A \subseteq M$.  If we write ``$f : M \xrightarrow[A]{} N$"  we assume that
$f$ is a $\K$-embedding and that $f\rest_A=\id_{A}$.
\end{itemize}
\end{nota}

Let us recall the following three properties. They are satisfied by all the classes considered in this paper, although not every AEC satisfies them.

\begin{defin}\
\begin{enumerate}
\item $\K$ has the \emph{amalgamation property} if for every $M, N, R \in \K$ such that $M \lea N, R$,
there is $R^* \in \K$ with $R \lea R^*$ and a $\K$-embedding  
$f: N \xrightarrow[M]{} R^*$.
\item  $\K$ has the \emph{joint embedding property} if for every $M, N \in \K$,
there is $R^* \in \K$ with $N \lea R^*$ and a $\K$-embedding  
$f: M \rightarrow R^*$.
\item  $\K$ has \emph{no maximal models} if for every $M \in \K$, there is $M^* \in \K$ such that $M <_{\K} M^*$.
\end{enumerate}

\end{defin}

  In \cite{sh300} Shelah introduced a notion of semantic type. The
original definition was refined and extended by many authors who
following \cite{grossberg2002} call
these semantic types Galois-types (Shelah recently named them orbital
types).
We present here the modern definition and call them Galois-types
throughout the text. We follow the notation of \cite[2.5]{mv}.

\begin{defin}\label{gtp-def}
  Let $\K$ be an AEC.
  
  \begin{enumerate}
    \item Let $\K^3$ be the set of triples of the form $(\bb, A,
N)$, where $N \in \K$, $A \subseteq |N|$, and $\bb$ is a sequence
of elements from $N$. 
    \item For $(\bb_1, A_1, N_1), (\bb_2, A_2, N_2) \in \K^3$, we
say $(\bb_1, A_1, N_1)E_{\text{at}}^{\K} (\bb_2, A_2, N_2)$ if $A
:= A_1 = A_2$, and there exists $f_\ell : N_\ell \xrightarrow[A]{} N$ $\K$-embeddings such that
$f_1 (\bb_1) = f_2 (\bb_2)$ and $N \in \K$.
    \item Note that $E_{\text{at}}^{\K}$ is a symmetric and
reflexive relation on $\K^3$. We let $E^{\K}$ be the transitive
closure of $E_{\text{at}}^{\K}$.
    \item For $(\bb, A, N) \in \K^3$, let $\gtp_{\K} (\bb / A;
N) := [(\bb, A, N)]_{E^{\K}}$. We call such an equivalence class a
\emph{Galois-type}. Usually, $\K$ will be clear from the context and we will omit it.
\item For $M \in \K$, $\gS_{\K}(M)= \{  \gtp_{\K}(b / M; N) : M
\leq_{\K} N\in \K \text{ and } b \in N\} $ 
\item For $\gtp_{\K} (\bb / A; N)$ and $C \subseteq A$, $\gtp_{\K} (\bb / A; N)\upharpoonright_{C}:= [(\bb, C, N)]_E$.

  \end{enumerate}

\end{defin}

\begin{remark}\label{trans}
If $\K$ has amalgamation, it is straightforward to show that $E_{\text{at}}^{\K}$ is
transitive.
\end{remark}

\begin{defin}
 An AEC is \emph{$\lambda$-Galois-stable}  if for any $M \in
\K_\lambda$, $| \gS_{\K}(M) | \leq \lambda$.

\end{defin}

The following notion was isolated by  Grossberg and VanDieren in \cite{tamenessone}.

\begin{defin} 
$\K$ is \emph{$(< \kappa)$-tame} if for any $M \in \K$ and $p \neq q \in \gS(M)$,  there is $A \subseteq M$ such that $|A |< \kappa$ and $p\upharpoonright_{A} \neq q\upharpoonright_{A}$. $\K$ is \emph{$\kappa$-tame} if it is $(< \kappa^+)$-tame.
\end{defin}

Let us recall the following concept that was introduced in
\cite{kosh}.

\begin{defin} Let $\K$ be an AEC.
$M$ is $\lambda$-\emph{universal over} $N$ if and only if $N \lea M$
 and for any $N^* \in \K_{\leq\lambda}$ such that
$N \lea N^*$, there is $f: N^* \xrightarrow[N]{} M$. $M$ is \emph{universal
over} $N$ if and only if $\| N\|= \| M\| $ and $M$ is $\| M
\|$-\emph{universal over} $N$. 
\end{defin} 

The next fact gives conditions for the existence of universal extensions.

\begin{fact}[{\cite[\S II]{shelahaecbook}, \cite[2.9]{tamenessone}}]\label{uni} Let $\K$ an AEC with joint embedding, amalgamation and no maximal models. If $\K$ is $\lambda$-Galois-stable, then for every  $P
\in \K_\lambda$, there is $M \in \K_\lambda$ such that $M$ is
universal over $P$.
\end{fact}

The following notion was introduced in \cite{kosh} and plays an important role in this paper.

\begin{defin}\label{limit}
Let $\alpha < \lambda^+$ a limit ordinal.  $M$ is a \emph{$(\lambda,
\alpha)$-limit model over} $N$ if and only if there is $\{ M_i : i <
\alpha\}\subseteq \K_\lambda$ an increasing continuous chain such
that $M_0 :=N$, $M_{i+1}$ is universal over $M_i$ for each $i <
\alpha$ and $M= \bigcup_{i < \alpha} M_i$. We say that $M\in
\K_\lambda$ is a $(\lambda, \alpha)$-limit model if there is $N \in
\K_\lambda$ such that $M$ is a $(\lambda, \alpha)$-limit model over
$N$. We say that $M\in \K_\lambda$ is a limit model if there is
$\alpha < \lambda^+$ limit such that $M$  is a $(\lambda,
\alpha)$-limit model.

\end{defin}

Observe that by iterating Fact \ref{uni} there exist limit models in Galois-stability cardinals for AECs with joint embedding, amalgamation and no maximal models.

In this paper, we deal with the classical global notion of
universal model.

\begin{defin}
Let $\K$ an AEC and $\lambda$ a cardinal. $M \in \K$ is a \emph{universal model in
$\K_\lambda$} if $M \in \K_\lambda$ and if given any $N \in \K_\lambda$, there is $f: N \to M$.
\end{defin}

\begin{remark}\label{easy-u}
When an abstract elementary class has joint embedding, then  $M$ is
universal over $N$ or $M$ is a limit model implies that $M$ is a
universal model in $\K_{\| M \|}$. A proof is given in 
\upshape{\cite[2.10]{maz}}.
\end{remark}

\subsection{Model theory of modules} 
For most of the basic results of the model theory of modules, we use 
the comprehensive text \cite{prest} of M.\ Prest as our primary 
source. The detailed history of these results can be found there.

The following definitions are
fundamental and will be used throughout the text.

\begin{defin} Let $R$ be a ring and $L_{R}= \{0, +,
-\} \cup \{ r  : r \in R \}$ be the language of $R$-modules.
\begin{itemize}
\item $\phi(\bar{v})$ is a \emph{$pp$-formula} if and only if
\[\phi(\bar{v})= \exists w_1  ... \exists w_l (\bigwedge_{j=1}^{m}
\Sigma_{i=1}^{n} r_{i,j} v_i + \Sigma_{k=1}^{l} s_{k,j}w_{k}=0 ),\]
 where $r_{i,j}, s_{k,j} \in R$ for every $i \in \{ 1,..., n\}, j \in \{1,...,m\}, k\in \{1,...,l\}$. 
\item Given $N$ an $R$-module, $A \subseteq N$ and $\bar{b} \in N^{<
\omega}$ we define the \emph{pp-type} of $\bar{b}$ over $A$ in $N$ as \[pp(\bar{b}/A, N)=
\{\phi(\bar{v}, \bar{a}) : \phi(\bar{v}, \bar{w}) \text{ is a
pp-formula, } \bar{a} \in A \text{ and }  N\models \phi[\bar{b},
\bar{a}] \}.\]

\item Given $M, N$ $R$-modules we say that $M$ is a \emph{pure
submodule} of $N$, written as $M \leq_{pp} N$, if and only if $M
\subseteq N$ and $pp(\bar{a}/ \emptyset, M)=  pp(\bar{a}/ \emptyset,
N)$ for every $\bar{a} \in M^{<\omega}$. Observe that in particular
if $M \leq_{pp} N$ then $M$ is a submodule of $N$.
\end{itemize}
\end{defin}

A key property of $R$-modules is that they have $pp$-quantifier
elimination, i.e., every formula in the language of $R$-modules is
equivalent to a boolean combination of $pp$-formulas.

\begin{fact}[Baur-Monk-Garavaglia, see e.g.
{\cite[\S 2.4]{prest}}]\label{ppq} Let $R$ be a ring
and $M$ a \tupop left\tupcp\ $R$-module. Every formula in the
language of
$R$-modules is equivalent modulo $Th(M)$ to a boolean combination of pp-formulas.

\end{fact}

The above theorem makes the model theory of modules algebraic in 
character,  and we will use many of its
consequences throughout the text. See for example  Facts \ref{bool},
\ref{ineq}, \ref{pp=tp} and \ref{equiv}.

Recall that given $T$ a complete first-order theory and $A \subseteq M$ with $M$ a model of $T$, $S^T(A)$ is the set of complete first-order types with parameters in $A$. A complete first-order theory $T$ is \emph{$\lambda$-stable} if $|S^T(A)| \leq \lambda$ 
 for every $A \subseteq M$ with $|A|=\lambda$ and $M$ a model of $T$. For a complete first-order theory $T$ this is equivalent to $(Mod(T), \preceq)$ being $\lambda$-Galois-stable, where $ \preceq$ is the elementary substructure relation.

\begin{fact}[Fisher, Baur, see e.g. {\cite[3.1]{prest}}]\label{st-c}
If $T$ is a complete first-order theory extending the theory of
$R$-modules and $\lambda^{|T|} = \lambda$, then $T$ is
$\lambda$-stable.  
\end{fact}

Pure-injective modules generalize the notion of injective module.

\begin{defin}
A module $M$ is \emph{pure-injective} if and only if for every module
$N$, if $M \leq_{pp} N$ then $M$ is a direct summand of $N$. 
\end{defin}

There are many statements equivalent to the definition of
pure-injectivity. The following equivalence will be used in the last
section:

\begin{fact}[{\cite[2.8]{prest}}]\label{epi} Let $M$ be an
$R$-module. The following are equivalent:
\begin{enumerate}
\item $M$ is pure-injective.
\item Every $M$-consistent pp-type $p(x)$ over $A \subseteq M$ with $|A|
\leq |R| + \aleph_0$, is realized in $M$.\footnote{For an incomplete
theory $T$ we say that a  pp-type $p(x)$ over $A \subseteq M$ is
$M$-consistent if it is realized in an elementary extension of $M$.}
\end{enumerate}
\end{fact}

That is, pure-injective modules are saturated with respect to 
$pp$-types. They often suffice as a substitute for saturated models 
in the model theory of modules.

 We will also use the pure hull of a module. The next
fact has all the information the reader will need about them. They
are studied extensively in \cite[\S 4]{prest} and \cite[\S
3]{ziegler}.
\begin{fact}\label{pih} \
\begin{enumerate}

\item For $M$ a module the \emph{pure hull of $M$}, denoted by
$\overline{M}$, is a pure-injective module such that $M \leq_{pp}
\overline{M}$ and it is minimum with respect to this. Its existence
follows from {\upshape \cite[3.6]{ziegler}} and the fact that every module can be
embedded in a pure-injective module.
\item {\upshape\cite{sab}} For $M$ a module, $M \preceq
\overline{M}$. 
\end{enumerate}
\end{fact}

\subsection{Torsion-free groups} The following class will be studied
in detail.

\begin{defin}
Let $\K^{tf}=(K^{tf}, \leq_{pp})$ where $K^{tf}$ is the class of
torsion-free abelian groups in the language $L_{\mathbb{Z}}= \{0, +,
-\} \cup \{ z : z \in \mathbb{Z} \}$ \tupop the usual language of
$\mathbb{Z}$-modules\tupcp and $\leq_{pp}$ is the pure
subgroup relation. Recall that $H$ is a pure subgroup of $G$ if for
every $n \in \mathbb{N}$, $nG \cap H  = nH$.
\end{defin}

It is known that $\K^{tf}$ is an AEC with
$\LS(\K^{tf})=\aleph_0$ that has joint embedding, amalgamation and no
maximal models (see \cite{grp}, \cite{baldwine} or \cite[\S 4]{maz}).
Furthermore limit models of uncountable cofinality were described in
\cite{maz}.

\begin{fact}[{\cite[4.15]{maz}}]\label{purec} 
If $G \in \K^{tf}$ is a $(\lambda, \alpha)$-limit model and
$\cof(\alpha) \geq \omega_1$, then \[G \cong \mathbb{Q}^{(\lambda)}
\oplus \prod_{p} \overline{\mathbb{Z}_{(p)}^{(\lambda)}}. \]
\end{fact}

\section{Universal models in classes of $R$-modules}

In this section we will construct universal models for certain
classes of $R$-modules.

\begin{nota}\label{not} Given $R$ a ring, we denote by
$\textbf{Th}_R$ the theory of left $R$-modules. 
Given $T$ a first-order theory \tupop not necessarily complete\tupcp\
extending
the theory of \tupop left\tupcp\ $R$-modules, let $\K^{T}= ( Mod(T),
\leq_{pp})$
and $|T|=|R|+\aleph_0$.

\end{nota}

Our first assertion will be that $\K^T$ is always an abstract
elementary class. In order to prove this, we will use the following
two corollaries of $pp$-quantifier elimination (Fact \ref{ppq}).
Given $n\in\mathbb{N}$ and $\phi, \psi$ $pp$-formulas such that
$\textbf{Th}_R \vdash \psi \to \phi$ we denote by $\Inv(-, \phi, \psi)
\geq n$ the first-order sentence satisfying: $M\models \Inv(-,
\phi, \psi) \geq n$ if and only if $[ \phi(M) : \psi(M)] \geq n$. 
Such a formula is called an \emph{invariant condition}.

\begin{fact}[{\cite[2.15]{prest}}]\label{bool} Every sentence in the
language of $R$-modules is equivalent, modulo the theory of
$R$-modules, to a boolean combination of invariant conditions.
\end{fact}

\begin{fact}[{\cite[2.23(a)(b)]{prest}}]\label{ineq} Let $M$, $N$
    be 
$R$-modules and $\phi, \psi$ $pp$-formulas such that $\textbf{Th}_R
\vdash \psi \to \phi$.
\begin{enumerate}
\item If $M \leq_{pp} N$, then $\Inv(M,\phi, \psi) \leq \Inv(N, \phi,
\psi)$.
\item $\Inv(M \oplus N, \phi, \psi)= \Inv(M,\phi, \psi) \Inv(N, \phi,
\psi)$.
\end{enumerate}

\end{fact}

\begin{lemma}
If $T$ is a first-order theory extending the theory of $R$-modules,
then $\K^T$ is an abstract elementary class with
$\LS(\K^T)=|T|$.
\end{lemma}
\begin{proof}
It is easy to check that $\K^T$ satisfies all the axioms of an AEC
except possibly the Tarski-Vaught axiom.
 Moreover if $\delta$ is a limit ordinal,
$\{ M_i \in \K^T : i < \delta \}$ is an increasing chain (with
respect to $\leq_{pp}$) and $N \in \K^T$ such that $\forall i <
\delta( M_i \leq_{pp} N)$, then $\forall i < \delta( M_i \leq_{pp}
M_\delta = \bigcup_{i < \delta} M_i \leq_{pp} N)$. Therefore, we only
need to show that if $\delta$ is a limit ordinal and $\{ M_i \in \K^T
: i < \delta \}$ is an increasing chain, then $M_\delta$ is a model of $T$.

First, by Fact \ref{bool}, every $\sigma \in T$ is
equivalent modulo $\textbf{Th}_R$ to a boolean combination of
invariant conditions. By putting that boolean combination in
conjunctive
normal form and separating the conjuncts we conclude that:
\[Mod(T)=Mod(\textbf{Th}_R \cup \{ \theta_{\beta} : \beta < \alpha\}), \]
where $\alpha \leq |T|$ and each \( \theta_{\beta}  \) is a finite 
disjunction of invariants statements of the form
\(  \Inv(- , \phi,\psi) \geq k \) or of the form
\( \Inv(- , \phi,\psi) < k \) (for some \emph{pp}-formulas
\( \phi,\ \psi \) such that
\( \textbf{Th}_R \vdash \psi \to \phi \) and some positive integer 
\( k \)).

Let $\delta$ be a limit ordinal and $\{ M_i \in \K^T : i < \delta \}$ an
increasing chain. It is clear that $M_\delta \models \textbf{Th}_R$ 
and that \( M_{i}\leq_{pp} M_\delta \) for all \( i<\delta \)\,.
Take $\beta < \alpha$ and consider \( \theta_{\beta} \)\,. There 
are two cases:

\underline{Case 1:} Some disjunct of  \( \theta_{\beta} \) is of the 
form \(  \Inv(- , \phi,\psi) \geq k \) and for some \( i<\delta \)\,,
\( M_{i}\models \Inv(- , \phi,\psi) \geq k \)\,. Since
\( M_{i}\leq_{pp} M_\delta \)\,, by Fact \ref{ineq}.(1) it follows 
that $\Inv(M_i , \phi, \psi) \leq  \Inv(M_\delta, \phi, \psi)$\,, 
and so $M_\delta \models\theta_\beta$. 

\underline{Case 2:} Every disjunct of \( \theta_{\beta} \) satisfied 
by a \( M_{i} \), for \( i<\delta \)\,,
is of the form \( \Inv(- , \phi,\psi) < k \) (for some
\( \phi,\ \psi, \mbox{ and }k \)). Since \( \delta \) is a limit 
ordinal and \( \theta_{\beta} \) is a finite disjunction, there is some 
cofinal subchain \( \{M_{i'}\} \) of
$\{ M_i  : i < \delta \}$\,,  such that each \( M_{i'} \) satisfies the 
same disjunct of \( \theta_{\beta} \)\,. So without loss of 
generality we can assume that this is true of the entire chain, i.e, there 
are \( \phi,\ \psi, \mbox{ and }k \) such that 
\( M_{i}\models \Inv(- , \phi,\psi) < k  \) for all \( i<\delta \)\, and $ \Inv(- , \phi,\psi) < k$ is a disjunct of $\theta_\beta$.
A counterexample to \( \Inv(M_{\delta}, \phi,\psi) < k \) would be 
witnessed by finitely many tuples from \( M_{\delta} \)\,, hence by 
finitely many tuples from \( M_{i} \)\, for some \( i<\delta \)\,, a 
contradiction. Therefore, $M_\delta \models \theta_\beta$.

\end{proof}

\begin{remark}
If $T$ has an infinite model, then $\K^T$ has no maximal
models. An infinite model \( M \) of \( T \) has arbitrarily
large 
elementary extensions, which are, \emph{ipso facto},  models of 
\( T \) and pure extensions of \( M \).
\end{remark}

The reader might wonder if $\K^T$ satisfies any other of the
structural properties of an AEC such as joint embedding or
amalgamation. We show that if $\K^T$ is closed under
direct sums, then $\K^T$ has both of these properties. This will be done in three steps.

\begin{fact}[{\cite[Exercise 1, \S 2.6]{prest}}]\label{weakap} Let $M, N_1, N_2 \in \K^T$.
If $M \leq_{pp} N_1$ and $M \preceq N_2$, then there are $N \in \K^T$ and
$f: N_1 \xrightarrow[M]{} N$ with $f$ elementary embedding and $N_2 \leq_{pp} N$.
\end{fact}

\begin{proof}[Proof sketch]
Introduce new distinct constant symbols for the elements of \( N_{1} \) and 
 \( N_{2} \)\,, agreeing on their common part 
\( M \)\,. Let \( \Delta(N_{1}) \) be the (complete) elementary diagram of 
\( N_{1} \)\,, let \( p^{+}(N_{2})=\{\phi(\overline{a}) : \phi\text{ is a 
\emph{pp}-formula, } \bar{a} \in N_2^{<\omega} \text{ and }N_{2}\models\phi[\overline{a}]\} \)\,, and let
\( p^{-}(N_{2})=\{\neg\phi(\overline{a}) : \phi\text{ is a 
\emph{pp}-formula, } \bar{a} \in N_2^{<\omega} \text{ and }N_{2}\models\neg\phi[\overline{a}]\} \)\,. Then it is straightforward to verify that 
\[
\Sigma= \Delta(N_{1})\cup p^{+}(N_{2})\cup p^{-}(N_{2}) 
\]
is finitely satisfiable in \( N_{1 } \) and any model \( N \) of 
\( \Sigma \) has the desired properties. \end{proof}

\begin{prop}\label{pihap}
If $\K^T$ is closed under direct sums, then pure-injective modules are amalgamation bases\footnote{Recall that $N \in \K$ is an \emph{amalgamation base}, if given $N \lea N_1, N_2 \in \K$, there are $L \in \K$ and $f: N_2 \xrightarrow[M]{} L$ such that $N_1 \lea L$.} .
\end{prop}
\begin{proof}
Let $N \leq_{pp} N_1, N_{2}$ all
in $\K^T$ with $N$ pure-injective. Since $N$ is pure-injective there
are submodules $M_1, M_2$  of  $N_1, N_2$ respectively,
such that for $l \in \{1,2 \}$ we have that
$N_l = N \oplus M_l$. Let 
$L =  N_{1}\oplus N_{2} =(N \oplus M_1) \oplus (N \oplus M_2)$.
Since $\K^{T}$ is closed under direct sums $L
\in \K^{T}$. Define
$f_1: N_1 \to L$ by $f_{1}(n,m_{1})=(n, m_{1}, n, 0)$ and 
$f_2: N_2 \to L$ by $f(n,m_{2})=(n, 0, n, m_{2})$. 
Clearly $f_1, f_2$ are pure embeddings with
$f_1\rest_{N}= f_2\rest_{N}$. \end{proof}

\begin{lemma}\label{prod-ap}
If $\K^T$ is closed under direct sums, then:
\begin{enumerate}
\item  $\K^T$ has joint embedding.
\item $\K^T$ has amalgamation.
\end{enumerate}
\end{lemma}
\begin{proof}
For the joint embedding property observe that given $M, N \in \K^T$,
they embed purely in $M\oplus N$ which is in $\K^T$ by hypothesis.

Regarding the amalgamation property, let $M \leq_{pp} N_1, N_{2}$ all
in $\K^T$.  For $\ell \in \{1,2\}$, $M, N_\ell, \overline{M}$ satisfy
the hypothesis of Fact \ref{weakap}, since $M \preceq \overline{M}$ by
Fact \ref{pih}.(2).  Then for $\ell \in \{1,2\}$, there are
$N_\ell^*\in \K^T$ and $f_\ell: N_\ell \xrightarrow[M]{} N_\ell^*$, with
$f_\ell$ an elementary embedding and $\overline{M} \leq_{pp}
N_\ell^*$.

Since $\overline{M} \leq_{pp} N_1^*, N_2^*$ and $\overline{M}$ is 
pure-injective by Fact \ref{pih}.(1), it follows from Proposition
\ref{pihap} that there are $N \in \K^T$, $g_1: N_1^* \to N$ and 
$g_2:N_2^* \to N$ with 
$g_1\rest_{\overline{M}} = g_2\rest_{\overline{M}}$
and $g_1,g_2$ both $\K^T$-embeddings. Finally, observe that $g_1 \circ f_1: N_1 \to N$ and $g_2 \circ f_2:
N_2 \to N$ are $\K^T$-embeddings such that $g_1 \circ f_1\rest_M=g_2
\circ f_2\rest_M$. \end{proof}

From the algebraic perspective the natural 
   hypothesis is to assume that $\K^T$ is closed under direct sums. On the other hand, from the model theoretic perspective it is more natural
to assume that $\K^T$ has joint embedding and amalgamation. This is
always the case when $T$ is a complete theory, which is precisely Example \ref{ex}.(2) below. 

Since we just showed that in $\K^T$ closure under direct sums implies joint embedding and amalgamation, we will assume these throughout the paper.  

\begin{hypothesis}\label{hyp}
Let $R$ be a ring and $T$ a first-order theory \tupop not necessarily complete\tupcp\ with an infinite model
extending the theory of $R$-modules such that:
\begin{enumerate}
\item $\K^T$ has joint embedding.
\item $\K^T$ has amalgamation. 
\end{enumerate}
\end{hypothesis}

 Even after this discussion the reader might wonder if there are any
natural classes that satisfy the above hypothesis. We give some examples:

\begin{example}\label{ex}\
\begin{enumerate}
\item $\K^{tf} = (K^{tf}, \leq_{pp})$ where $K^{tf}$ is the class of
torsion-free abelian groups. In this case $T$ 
is a first-order axiomatization of torsion-free abelian groups.
Since torsion-free abelian groups are closed under direct sums, 
by Lemma \ref{prod-ap} $\K^{tf}$ has   joint embedding and
amalgamation.

\item $\K^T= (Mod(T), \leq_{pp})$ where $T$ is a complete theory
extending $\textbf{Th}_R$. This follows from the fact that if $M, N
\models T$, then $M \leq_{pp} N$ if and only if $M \preceq N$ by
pp-quantifier elimination.

\item $\K^{\textbf{Th}_R}= (Mod(\textbf{Th}_R), \leq_{pp})$. It is
clear that $\K^{\textbf{Th}_R}$ is closed under direct sums, 
so by Lemma \ref{prod-ap} $\K^{\textbf{Th}_R}$ has   joint embedding and
amalgamation.

\item $\K= (\chi, \leq_{pp})$ where $\chi$ is a definable category of modules in the sense of \cite[\S 3.4]{prest09}. In this case $T = \{ \forall x (\phi(x) \to \psi(x)) :  \textbf{Th}_R \vdash \psi \to \phi \text{ and } \phi(M)=\psi(M) \text{ for every } M \in \chi\}$ and $\K$ has joint embedding and amalgamation because $\K$ is closed
under direct sums \tupop by {\upshape \cite[3.4.7]{prest09}}\tupcp\
and by Lemma \ref{prod-ap}.

\item  $\K=(\textbf{C}, \leq_{pp})$ where $\textbf{C}$ is a universal
Horn class. In this case $T=T_{\textbf{C}}$ 
\tupop where $T_{\textbf{C}}$
is an axiomatization of $\textbf{C}$ \tupcp\
and $\K$ has joint embedding and amalgamation because $\K$ is closed
under direct sums \tupop by {\upshape \cite[15.8]{prest}}\tupcp\
and by Lemma \ref{prod-ap}.

\item $\K=(\eff_r, \leq_{pp})$ where $r$ is a radical of finite type and $\eff_r$ is the class of $r$-torsion-free modules.
In this case $T$ exists by {\upshape \cite[15.9]{prest}}
and $\K$ has joint embedding and amalgamation because $\K$ is closed
under direct sums \tupop by {\upshape \cite[15.8]{prest}}\tupcp\ and by Lemma \ref{prod-ap}.

\item $\K = (\tte_r, \leq_{pp})$ where $r$ is a left exact radical, $\tte_r$ is the class of $r$-torsion modules
and $\tte_r$ is closed under products. In this case $T$ exists by
{\upshape \cite[15.14]{prest}} and $\K$ has joint embedding and amalgamation by a similar reason to
\tupop 5\tupcp.

\item $\K = (K_{\text{flat}}, \leq_{pp})$ where $K_{\text{flat}}$ is
the class of \tupop left\tupcp\
flat $R$-modules over a right coherent ring. In this case $T$ exists
by  {\upshape \cite[14.18]{prest}} and $\K$ has joint embedding and amalgamation because the 
class of  flat modules is
closed under direct sums and by Lemma \ref{prod-ap}.

\end{enumerate}

\end{example}



The following example shows that Hypothesis \ref{hyp} is not trivial,
i.e., given $T$ a first-order theory with an infinite model extending
the theory of $R$-modules Hypothesis \ref{hyp}  does not necessarily
hold.

\begin{example}

Let $T = \textbf{Th}_{\mathbb{Z}} \cup \{  \Inv(-, x=x, 3x=0) < 6 \}$.

Let \( A \) be an abelian group satisfying \( T \) 
and \( B \) the subgroup of \( A \) defined by \( 3x=0 \)\,. Then
\( |A/B|\in\lbrace 1,2,3,4,5\rbrace \) and so 
\( A/B\cong A_{0} \)\,, where \( A_{0} \) is one of the finite groups
\(\lbrace0\rbrace,\,\mathbb{Z}/2,\,
\mathbb{Z}/2\times \mathbb{Z}/2,\,\mathbb{Z}/4, \,
\mathbb{Z}/5 \)\,, or \( \mathbb{Z}/3 \)\,.

In particular, if \( B=0 \)\,, observe that
the first five  $A_{0}$'s just listed are models of \( T \). On the other hand, if \( B\ne 0 \)\,, then 
\( B\cong (\mathbb{Z}/3)^{(\kappa)}\) for some finite or infinite 
cardinal \( \kappa \)\,, and since 3 is a prime, it has no 
non-trivial extensions by any of the groups \( A_{0} \)\,. There is 
one exceptional case, as \( \mathbb{Z}/9 \) 
is an extension of \( \mathbb{Z}/3 \) by itself. 

Since the invariants 
multiply across direct sums \tupop Fact \ref{ineq}\tupcp, then  
all the models of \( T \) are  \( \mathbb{Z}/9 \) 
or of the form \( A_0 \) or
\( (\mathbb{Z}/3)^{(\kappa)}\oplus A_{0} \)\,, for some choice of
\( A_{0} \) and \(\kappa \) a finite or infinite cardinal. 

Therefore, there are many examples of failures of the joint embedding property: 
amongst them we have that \( \mathbb{Z}/2 \) and \(\mathbb{Z}/5 \) do not have a common
extension to a model of \( T \)\,, and since the zero module is 
pure-injective, this is an example of the failure of amalgamation over 
pure-injectives. Since \( (\mathbb{Z}/3)^{(\aleph_{0})} \) is 
pure-injective, 
\( (\mathbb{Z}/3)^{(\aleph_{0})} \oplus \mathbb{Z}/2\) and
\( (\mathbb{Z}/3)^{(\aleph_{0})} \oplus \mathbb{Z}/5 \) provide an 
infinite example.
\end{example}

It is worth pointing out that there is an easy first-order argument  to find
universal models if one assumes the hypothesis that $\K^T$ is closed
under direct sums.\footnote{This was discovered after we had a proof using the theory of abstract elementary classes \tupop see Lemma \ref{universal1} \tupcp.}

\begin{lemma}
If $\K^T$ is closed under direct sums and $\lambda^{|T|}=\lambda$,
then $\K^T_\lambda$ has a universal model.
\end{lemma}
\begin{proof}
    Observe that \( T  \) has no more than \( 2^{|T|} \) complete 
    extensions. Each such extension is  $\lambda$-stable, see Fact
    \ref{st-c}, and so has a saturated model of cardinality 
    \( \lambda \)\,. Take the direct sum \( U \) of all of these; it 
    has cardinality \( 2^{|T|}\lambda=\lambda \)\,. We claim that 
    \( U \in \K^T_\lambda \) and is universal in 
    \( \K^T_\lambda \)\,. But \( \K^{T} \) is closed under direct sums, 
    so \( U \in \K^T \); and we have already observed that
    \( \|U\|=\lambda \)\,.
    
    If $N \in \K^T_\lambda$\,, then \( N \) is elementarily embedded 
    in the \( \lambda \)-saturated model 
    of \( \mathrm{Th}(N) \)  which is a summand of \( U \)\,, and 
    hence \( N \) is purely embedded in \( U \)\,. \end{proof}

\subsection{Galois-stability}

The following consequence of $pp$-quantifier elimination will be the
key to the arguments in this subsection:

\begin{fact}[{\cite[2.17]{prest}}]\label{pp=tp}
Let $N \in \K^{T}$, $A \subseteq N$ and $\bar{b}_{1}, \bar{b}_{2} \in
N^{<\omega}$. Then: 
\[ \pp(\bar{b}_{1}/A , N) = \pp(\bar{b}_{2}/A , N) \text{ 
if
and only if } \tp{\bar{b}_{1}}{A}{ N} = \tp{\bar{b}_{2}}{A}{N}.\] 
\end{fact}

With this, we are able to show that $pp$-types and Galois-types are the same over models.
\begin{lemma}\label{pp=gtp}
Let $M, N_1, N_2 \in \K^T$,  $M \leq_{pp} N_1, N_2$, $\bar{b}_{1}
\in  N_1^{<\omega}$  and $\bar{b}_{2} \in N_2^{<\omega}$. Then:
 \[ \gtp(\bar{b}_{1}/M; N_1) = \gtp(\bar{b}_{2}/M; N_2) \text{ if
and
only if } \pp(\bar{b}_{1}/M , N_1) = \pp(\bar{b}_{2}/M, N_2).\]
\end{lemma}
\begin{proof}
\underline{$\to$:} Suppose $\gtp(\bar{b}_{1}/M; N_1 ) =
\gtp(\bar{b}_{2}/M; N_2)$. Since $\Ka$ has amalgamation, there are
$N \in \Ka$ and $f: N_1 \to N$ a $\K^{T}$-embedding such that
$f\rest_{M}= \id_{M}$, $f(\bar{b}_{1})=\bar{b}_{2}$ and $N_2 \leq_{pp}
N$. Then the result follows from the fact that $\K^{T}$-embeddings
preserve and reflect $pp$-formulas by definition.

\underline{$\leftarrow$:} Suppose $pp(\bar{b}_{1}/M , N_1) =
pp(\bar{b}_{2}/M , N_2)$. Since $M \in \Ka$ and $\K^T$ has amalgamation,
there is $N \in \Ka$ and  $f: N_1 \to N$ a $\K^T$-embedding such that
$f\rest_{M}= \id_{M}$ and $N_2 \leq_{pp} N$. Using that
$\K^{T}$-embeddings preserve $pp$-formulas we have that
$\pp(f(\bar{b}_{1})/M , N)= \pp(\bar{b}_{2}/M , N)$.

Then by Fact \ref{pp=tp} it follows that $\tp{f(\bar{b}_{1})}{M}{ N}=
\tp{\bar{b}_{2}}{M}{ N}$. Let $N^*$ an elementary extension of $N$
such
that there is  $g\in Aut_{M}(N^*)$ with 
$g(f(\bar{b}_{1}))=\bar{b}_{2}$.
Observe that since $\K^{T}$ is first-order axiomatizable $N^* \in
\K^T$. Consider $h:= g \circ
f: N_1 \to N^{*}$.

It is clear that $h(\bar{b}_{1})=\bar{b}_{2}$, 
$h\rest_{M} = \id_{M}$ and
since being an elementary substructure is stronger than being a pure
substructure it follows that $h: N_1 \to N^{*}$ is a
$\K^T$-embedding and $N_2 \leq_{pp} N^{*}$. Therefore,
$\gtp(\bar{b}_{1}/M; N_1) = \gtp(\bar{b}_{2}/M; N_2)$. \end{proof}

The next corollary follows from the preceding lemma since we can witness that two
Galois-types are different by a $pp$-formula.

\begin{cor}\label{tame1} $\K^T$ is $(<\aleph_0)$-tame.
\end{cor}

The next theorem is the main result of this subsection.

\begin{theorem}\label{g-st}
If $\lambda^{|T|}=\lambda$, then $\K^{T}$ is $\lambda$-Galois-stable.
\end{theorem}
\begin{proof}
Let $M \in \K^T_\lambda$ and $\{ p_i : i < \alpha \}$ an enumeration without repetitions of $\gS(M)$ where $\alpha \leq 2^{\lambda}$. Since $\K^T$ has amalgamation, there is $N\in \K^T$ and $\{a_i  : i < \alpha\} \subseteq N$ such that $p_i = \gtp(a_i/ M; N)$ for every $i < \alpha$.

Let $\Phi: \gS(M) \to S^{Th(N)}_{pp}(M)$ be defined by
$p_i \mapsto \pp(a_i/M, N)$. 
By Lemma \ref{pp=gtp} $\Phi$ is a well-defined
injective function. By Fact \ref{pp=tp}
$|S^{Th(N)}_{pp}(M)|=|S^{Th(N)}(M)|$. 
Then it follows from Fact
\ref{st-c} that $|S^{Th(N)}(M)|\leq \lambda$, hence
$| \gS(M) |
\leq \lambda$. \end{proof}

\subsection{Universal models}  

It is straightforward to construct universal models in $\K^T$ for $\lambda$'s satisfying that $\lambda^{|T|}=\lambda$. This follows from Fact \ref{uni} and Remark \ref{easy-u}.

\begin{lemma}\label{universal1}
If $\lambda^{|T|}= \lambda$, then $\K^{T}_\lambda$ has a universal
model.
\end{lemma}

The following lemma shows how to build universal models in
cardinals where $\K^T$ might not be $\lambda$-Galois-stable.

\begin{lemma}\label{main3}
If $\forall \mu < \lambda( \mu^{|T|} < \lambda)$, then $\K^T_\lambda$
has a universal model.
\end{lemma}
\begin{proof} We may assume that $\lambda$ is a limit cardinal, because if it is not the case then we have that $\lambda^{|T|}=\lambda$ and we can apply Lemma \ref{universal1}.
Let $\cof(\lambda)= \kappa \leq \lambda$. By using the hypothesis that
$\forall \mu < \lambda( \mu^{|T|} < \lambda)$, it is easy to build
$\{ \lambda_i : i <\kappa \}$ an increasing continuous sequence of
cardinals such that $ \forall i (\lambda_{i + 1}^{|T|} = \lambda_{i+1})$ and $sup_{i < \kappa} \lambda_i = \lambda$.

We build $\{ M_i : i < \kappa \}$ an increasing continuous chain such that:
\begin{enumerate}
\item $M_{i +1}$ is $\| M_{i + 1} \|$-universal over $M_i$. 
\item $M_i \in \K_{\lambda_i}$. 
\end{enumerate}

 In the base step pick any $M \in \K^T_{\lambda_0}$ and if $i$ is limit, let $M_i = \bigcup_{j < i} M_j$. 

If $i = j + 1$, by construction we are given $M_j \in
\K^T_{\lambda_{j}}$. Using that $\K^T$ has no maximal models, we
find $N \in \K^T_{\lambda_{j+1}}$ such that $M_j \leq_{pp} N$. Since
$\lambda_{j + 1}^{|T|} = \lambda_{j +1}$, by Theorem \ref{g-st} $\K^T$ is $\lambda_{j+1}$-Galois-stable.
Then by Fact \ref{uni} applied to $N$, there is $M_{j
+1} \in \Ka_{\lambda_{j +1}}$ universal over
$N$. Using that $\K^T$ has amalgamation, 
it is
straightforward to check that (1) holds. 

This finishes the construction of the chain.

Let $M = \bigcup_{i < \kappa} M_i$. By (2)
$\| M \|
= \lambda$. We show that $M$ is universal in $\K_\lambda^T$.

Let $N \in \K^{T}_\lambda$ and $\{N_i : i < \kappa \}$ an
increasing continuous chain such that $\forall
i(  N_i \in \K^T_{\lambda_i})$ and $\bigcup_{i <
\kappa} N_i = N$.  We build  $\{  f_i : i < \kappa \}$ such that:
\begin{enumerate}

\item $f_i: N_i \to M_{i +1}$.
\item $\{ f_i : i < \kappa \}$ is an
increasing chain.
\end{enumerate}

Observe that this is enough by taking $f=
\bigcup_{i < \kappa} f_i : N = \bigcup_{i < \kappa} N_i \to
\bigcup_{i < \kappa} M_{i+1} =M $.

Now, let us do the construction. In this case the base step is
non-trivial. By joint embedding there is $g: N_0  \to M^*$ with $M_0
\leq_{pp} M^* \in \K^T_{\lambda_0}$. Now, since $M_1$ is $\| M_1
\|$-universal over $M_0$ there is $h: M^* \xrightarrow[M_0]{} M_1$. Let $f_0 :=
h \circ g$ and observe that this satisfies the requirements.

We do the induction steps.

If $i$ is limit, let $f_i = \bigcup_{j < i}
f_j: N_i= \bigcup_{j< i}N_j \to M_{i + 1}$.

If $i = j +1$, by construction we have $f_j : N_j \to M _{j + 1}$ and $N_j \leq_{pp} N_{j+1}$. Since $\K^T$ has amalgamation 
 there is $M' \in \K^{T}_{\lambda_{j
+1}}$ and $g: N_{j + 1} \to M'$ such that $M_{j+1} \leq_{pp} M'$ and $f_j\rest_{N_j}=g\rest_{N_j}$. Since  $M_{j + 2}$ is $\| M_{j +2}\|$-universal over $M_{j +1}$, there
is $h : M' \xrightarrow[M_{j +1}]{} M_{j + 2}$. Let $f_{j+1} := h \circ g$ and
observe that this satisfies the requirements.\end{proof}

Putting together Lemma \ref{universal1} and Lemma \ref{main3} we
get one of our main results.

\begin{theorem}\label{mainc}
 If $\lambda^{|T|}=\lambda$ or $\forall \mu < \lambda( \mu^{|T|} <
\lambda)$, then $\K^{T}_\lambda$ has a universal model.

\end{theorem}
The proof of Lemma \ref{universal1} and Lemma \ref{main3} can be
extended in a straightforward way to the following general setting. 

\begin{cor}\label{r-aec}
Let $\K$ be an AEC with joint embedding, amalgamation and no maximal
models. Assume there is $\theta_0 \geq \LS(\K)$ and $\kappa$ such
that for all $\theta \geq \theta_0$, if $\theta^{\kappa} =
\theta$, then $\K$  is $\theta$-Galois-stable.

Suppose $\lambda > \theta_0$. If $\lambda^{\kappa}=\lambda$ or
$\forall \mu < \lambda( \mu^{\kappa} < \lambda)$, then $\K_\lambda$
has a universal model.\footnote{In Lemma \ref{universal1} and Theorem
\ref{main3} $\theta_0= \LS(\K^T)=|T|$ and $\kappa=|T|$.}
\end{cor}

\begin{remark}
In {\upshape \cite[4.13]{vaseya}} it is shown that if $\K$ is an AEC
with joint
embedding, amalgamation and no maximal models, $\K$ is $\LS(\K)$-tame
and $\K$ is $\lambda$-Galois-stable for some $\lambda \geq \LS(\K)$,
then there are $\theta_0$ and $\kappa$ satisfying the hypothesis of
Corollary \ref{r-aec}.

\end{remark}

\subsection{Reduced torsion-free abelian groups}
Recall that $\K^{tf}$ has joint embedding and amalgamation, so it
satisfies Hypothesis \ref{hyp}. Moreover,  $|T^{tf}|=\aleph_0$,
therefore the next assertion follows directly from Theorem \ref{g-st}
and Theorem \ref{mainc}.

\begin{cor}\label{u-tf}\
\begin{enumerate}
\item  If $\lambda^{\aleph_0}=\lambda$, then $\K^{tf}$ is
$\lambda$-Galois-stable.
 \item If $\lambda^{\aleph_0}=\lambda$ or $\forall \mu < \lambda(
\mu^{\aleph_0} < \lambda)$, then $\K^{tf}_\lambda$ has a universal
model.
\end{enumerate}

\end{cor}

\begin{remark}\label{st-2}
In {\upshape \cite[0.3]{baldwine}} it is shown that: $\K^{tf}$ is
$\lambda$-Galois-stable if and only if  $\lambda^{\aleph_0}=\lambda$.
The argument given here differs substantially with that of
{\upshape \cite[0.3]{baldwine}}, their argument does not consider
$pp$-formulas and instead exploits the property that $\K^{tf}$ admits
intersections.
\end{remark}

As mentioned in the introduction, Shelah's result \cite[1.2]{sh820}
is concerned with reduced torsion-free groups instead of with
torsion-free groups. The next two assertion show how we can recover
his assertion from the above results. First let us introduce a new
class of groups.

\begin{defin}
Let $\K^{rtf}=(K^{rtf}, \leq_{pp})$ where $K^{rtf}$ is the class of
reduced torsion-free abelian groups defined in the usual
language $L_{\mathbb{Z}}$ of \( \mathbb{Z} \)-modules, 
and $\leq_{pp}$ is the pure
subgroup relation. Recall that a group $G$ is reduced if  its only
divisible subgroup is $0$.

\end{defin}

\begin{fact}\label{rtf} 
Let $\lambda$ an infinite cardinal. $\K^{tf}_\lambda$ has a universal
model if and only if $\K^{rtf}_\lambda$ has a universal model.
\end{fact}
\begin{proof}
The proof follows from the fact that divisible torsion-free abelian
groups of cardinality
$\le\lambda$ are purely embeddable into $\mathbb{Q}^{(\lambda)}$
and that every group can be written as a direct sum of a unique
divisible subgroup and a unique up to isomorphisms reduced subgroup
(see \cite[\S 4.2.4, \S 4.2.5]{fuchs}). \end{proof}

The following is precisely \cite[1.2]{sh820}.

\begin{cor}\label{she-res}
\begin{enumerate}\
\item If $\lambda^{\aleph_0}= \lambda$, then $\K^{rtf}_\lambda$ has a
universal model.
\item If $\lambda = \Sigma_{n < \omega} \lambda_n$ and $\aleph_0 \leq
\lambda_n = (\lambda_n)^{\aleph_0} < \lambda_{n+1}$, then
$\K^{rtf}_\lambda$ has a universal model.
\item $\K^{rtf}$ has amalgamation, joint embedding, is an AEC and is
$\lambda$-Galois-stable if $\lambda^{\aleph_0}=\lambda$.
\end{enumerate}
\end{cor}
\begin{proof}
For (1) and (2), realize that $\lambda$ either satisfies the first or
second hypothesis of Corollary \ref{u-tf}.(2), hence
$\K^{tf}_\lambda$ has a universal model. Then by Fact \ref{rtf} we
conclude that $\K^{rtf}_\lambda$ has a universal model in either case.

For (3), the first three assertions are easy to show. As for the last
one, this follows from Corollary \ref{u-tf}.(1) and the fact that if
$G, H \in \K^{rtf}$ and $a, b \in H$ then: $\gtp_{\K^{rtf}}(a/ G; H)=
\gtp_{\K^{rtf}}(b/ G; H)$ if and only if $\gtp_{\K^{tf}}(a/ G; H)=
\gtp_{\K^{tf}}(b/ G; H)$. \end{proof}

\begin{remark}\label{equal}
It is worth noticing that {\upshape Corollary \ref{u-tf}.(2)}
 not only implies
{\upshape \cite[1.2.1, 1.2.2]{sh820} (Corollary \ref{she-res}.(1) and
Corollary
\ref{she-res}.(2))}, but the two assertions are equivalent. The
backward direction follows from the fact that if $\lambda$ satisfies
 $\cof(\lambda) \geq \omega_1$ and $\forall \mu < \lambda(
\mu^{\aleph_0} < \lambda)$, then $\lambda^{\aleph_0}=\lambda$. 

\end{remark}

\begin{remark}
It follows from {\upshape Corollary \ref{u-tf}.(2)} that if $2^{\aleph_0}
<
\aleph_\omega$, then $\K^{tf}_{\aleph_\omega}$ has a universal model.
On the other hand, it follows from 
{\upshape \cite[3.7]{kojsh}} that if
$\aleph_\omega < 2^{\aleph_0}$, then $\K^{tf}_{\aleph_\omega}$ does
not have a universal model.  Hence the existence of a universal model
in $\K^{tf}$ of cardinality $\aleph_\omega$ is independent of 
\textup{ZFC}.
Similarly one can show that the existence of a universal model in
$\K^{tf}$ of cardinality $\aleph_n$ is independent of \textup{ZFC}
for every
$n \geq 1$.
\end{remark}

\section{Limit models in classes of $R$-modules}

In this section we will begin the study of limit models in classes of
$R$-modules under Hypothesis \ref{hyp}. The existence of limit models in $\K^T$ for
$\lambda$'s satisfying  $\lambda^{|T|}=\lambda$ follows directly
from Theorem \ref{g-st} and Fact \ref{uni}.

\begin{cor}\label{limo}
If $\lambda^{|T|}=\lambda$, then there is a $(\lambda, \alpha)$-limit
model in $\K^{T}$ for every $\alpha < \lambda^+$ limit ordinal. 
\end{cor}

We first show that any two limit models are elementarily equivalent. 
In order to do that, we will use one more consequence of $pp$-quantifier 
elimination \tupop Fact \ref{ppq}\tupcp.

\begin{fact}[ {\cite[2.18]{prest}}]\label{equiv} Let $M$ and $N$ 
$R$-modules.  $M$ is elementary equivalent to $N$ if and only if $\Inv(M, \phi, \psi) =
\Inv(N, \phi, \psi)$ for every $\phi, \psi$ $pp$-formulas in one free
variable such that $\textbf{Th}_R \vdash \psi \to \phi$.
\end{fact}

\begin{lemma}\label{limeq}
If $M, N$ are limit models, then $M$ and $N$ are elementary equivalent.
\end{lemma}
\begin{proof}
Assume $M$ is a $(\lambda, \alpha)$-limit model for $\alpha < \lambda^+$ and  let $\{M_i : i < \alpha \} \subseteq \K^T_\lambda$ be a witness for it. Similarly assume $N$ is a $(\mu,
\beta)$-limit model for $\beta < \mu^+$ and  let $\{N_i : i
< \beta \} \subseteq \K^T_\mu$ be a witness for it. 

By Fact \ref{equiv}, it is enough to show that for every 
$\phi, \psi$\,,  $pp$-formulas in one free
variable such that $\textbf{Th}_R \vdash \psi \to \phi$, and $n \in
\mathbb{N}$: $\Inv(M, \phi, \psi) \geq n$ if and only if $\Inv(N,
\phi, \psi) \geq n$. By the symmetry of this situation, we only need 
to prove one implication.
So consider such $pp$-formulas
 $\phi, \psi$ and $n \in \mathbb{N}$ such that 
$\Inv(M, \phi, \psi) \geq n$. 
 We show that $\Inv(N, \phi, \psi) \geq n$. 

 If $n=0$, the result is clear. So assume that $n\geq 1$. Then since  $\Inv(M, \phi, \psi) \geq n$, there are $m_0,...,m_{n-1} \in M$ such that:
\[M \models \bigwedge_{i} \phi(m_i) \wedge
\bigwedge_{i\neq j} \neg \psi(m_i - m_j).\]

Applying the downward L\"{o}wenheim-Skolem-Tarski axiom inside $M$ to $\{ m_i : i <
n\}$, we get $M^* \leq_{pp} M$ such that $M^* \in \K^{T}_{\LS(\K)}$
and $\{ m_i : i < n\} \subseteq M^*$.  Then it is still the case that
\[M^*
\models \bigwedge_{i} \phi(m_i) \wedge
\bigwedge_{i\neq j} \neg \psi(m_i - m_j). \]

By joint embedding there is $g$ and $M^{**}\in \K^T_\mu$ such that $g:
M^* \to M^{**}$ and $N_0 \leq_{pp} M^{**}$.  Then since $N_1$ is
universal over $N_0$, there is $h: M^{**} \xrightarrow[N_0]{} N_1$.  Finally,
observe that:
\[ N  \models \bigwedge_{i} \phi(h\circ g(m_i)) \wedge
\bigwedge_{i\neq j} \neg \psi(h\circ g(m_i) - h\circ g(m_j)).\]
Hence $\Inv(N, \phi, \psi) \geq n$.  \end{proof}

\begin{remark}
Observe that in the proof of the above lemma we only used that $\K^T$ is an AEC of modules with the joint embedding property.
\end{remark}

As in \cite[\S 4]{maz}, limit models with chains of big cofinality
are easier to understand than those of small cofinalities. Due to
this we begin by studying the former.

\begin{theorem}\label{bigpi}
Assume $\lambda\geq |T|^+=\LS(\K^T)^+$. If $M$ is a $(\lambda,
\alpha)$-limit model and $\cof(\alpha)\geq |T|^+$, then $M$ is
pure-injective.
\end{theorem}
\begin{proof}
Fix $\{ M_i : i < \alpha\}$ a witness to the fact that $M$ is a
$(\lambda, \alpha)$-limit model. We show that $M$ is pure-injective
using the equivalence of Fact \ref{epi}.

Let $p(x)$ be an $ M$-consistent $pp$-type over $A \subseteq M$ and $|A|
\leq |R| + \aleph_0 = |T|$.  Then
there is a module $N$ and $b \in N$ with $M \preceq N \in \K^T_{\| M
\|}$ and $b$
realizing $p$. Since $|A| \leq |T|$ and $\cof(\alpha) \geq |T|^+$,
there
is $i < \alpha$ such that $A \subseteq M_i$.

Note that $M_i \leq_{pp} N$.  Then there is $f: N \xrightarrow[M_i]{} M_{i+1}$,
because  $M_{i+1}$ is universal over
$M_i$. Since $A$ is fixed by the
choice of $M_i$, it is easy to see that $f(b) \in M_{i+1}\leq_{pp}
M$ realizes $p(x)$. Therefore, $M$ is pure-injective. \end{proof}

The following fact about pure-injective modules is a generalization
of Bumby's result \cite{bumby}. A proof of it (and a discussion of 
the general setting) appears in
\cite[3.2]{gks}.
We will use it to show uniqueness of limit models of big cofinalities.

\begin{fact}\label{ipi}
Let $M, N$ be pure-injective modules. If there is $f: M \to N$ a
$\K^{\textbf{Th}_R}$-embedding and $g: N \to M$ a
$\K^{\textbf{Th}_R}$-embedding, then $M \cong N$. 
\end{fact}

\begin{cor}\label{uniqb} Assume $\lambda\geq |T|^+=\LS(\K^T)^+$.
If $M$ is a $(\lambda, \alpha)$-limit model and $N$ is a $(\lambda,
\beta)$-limit model such that $\cof(\alpha), \cof(\beta) \geq |T|^+$,
then $M$ is isomorphic to $N$. 
\end{cor}
\begin{proof}
It is straightforward to check that $M$ and $N$ are universal models
in $\K^T_\lambda$ (see Remark \ref{easy-u}). Since $M$ and $N$ are pure-injective by Theorem \ref{bigpi}, then
the result  follows from Fact \ref{ipi} because $\K^T$-embeddings and $\K^{\textbf{Th}_R}$-embeddings  
are the same. \end{proof}

Dealing with limit models of small cofinality is complicated. We
will only be able to describe limit models of countable cofinality
under the additional assumption that $\K^{T}$ is closed under direct
sums. All the examples of Example \ref{ex}, except Example \ref{ex}.(2), satisfy this additional hypothesis.

\begin{lemma}\label{universal} Assume $\K^{T}$ is closed under
direct sums.
If $M \in \K^{T}_{\lambda}$ is pure-injective and $U \in
\K^{T}_{\lambda}$ is a universal model in $\K^{T}_{\lambda}$, then $M
\oplus U$ is universal over $M$.
\end{lemma}
\begin{proof}
It is clear that $M \leq_{pp} M \oplus U$ and that both modules have
the same cardinality, so  take $N \in \K^{T}_{\lambda}$  such that $M
\leq_{pp} N$. Since $M$ is pure-injective we have that $N = M \oplus
M'$ for some $M' \in \K^{T}_{\leq \lambda}$. Using that $\K^{T}$ has
no maximal models and that $U$ is universal in $\K^T_\lambda$, there
is $f': M'  \to U$ a $\K^{T}$-embedding. Let $f: M \oplus M' \to M
\oplus U$
be given by $f(a + b)=  a + f'(b)$. It is easy to check that $f$ is a
$\K^{T}$-embedding that fixes $M$. \end{proof}

\begin{theorem}\label{countablelim} Assume $\lambda\geq
|T|^+=\LS(\K^T)^+$ and $\K^T$ is closed under direct sums. If $M$ is
a $(\lambda, \omega)$-limit model and  $N$ is a
$(\lambda, |T|^+)$-limit model, then $M \cong  N^{(\aleph_0)}$.

\end{theorem}
\begin{proof}
 For every $ i < \omega$, let $N_i$ be given by $i$-many direct copies of $N$. Consider the increasing
chain $\{ N_i : i < \omega \} \subseteq \K^{T}_\lambda$. 

By Theorem \ref{bigpi} $N \in \K^{T}$  is pure-injective. Since
pure-injective modules are closed under finite direct sums, $N_i$ is
pure-injective for every $i < \omega$ . Moreover, for each $i <
\omega$, $N_{i+1}= N_i  \oplus N$ is universal over $N_i$ because $N$
is universal in $\K^{T}_\lambda$, $N_i$ is pure-injective and by
Lemma \ref{universal}. Therefore, $N_\omega := \bigcup_{i < \omega}
N_i$ is a $(\lambda, \omega)$-limit  model.

Since $N_\omega$ and $M$ are limit models with chains of the same
cofinality, a back-and-forth argument shows that $N_\omega \cong M$.
Hence $M \cong N^{(\aleph_0)}$. \end{proof}

Lemma \ref{universal} can also be used to characterize Galois-stability in classes closed under direct sums.

\begin{cor} Assume $\K^{T}$ is closed under
direct sums and $\lambda \geq |T|^+$ is an infinite cardinal.
$\K^T$ is $\lambda$-Galois-stable if and only if
$\K^T_\lambda$ has a pure-injective universal model.
\end{cor}
\begin{proof}
The forward direction follows from the fact that $(\lambda, |T|^+)$-limit models are pure-injective by Theorem \ref{bigpi}. So we sketch the backward direction. Let $M \in \Ka_\lambda$ and $U\in \Ka_\lambda$ a  pure-injective universal model. By universality of $U$ we may assume that $M \leq_{pp} U$. Then by minimality of the pure hull we have that $\overline{M} \leq_{pp} U$, thus $\overline{M} \in \Ka_\lambda$. So by Lemma \ref{universal} $\overline{M} \oplus U$ is universal over $\overline{M}$. Therefore, every type over $M$ is realized in $\overline{M} \oplus U$. Hence $| \gS(M) |\leq \| \overline{M} \oplus U \|=\lambda$.  \end{proof}

\begin{remark} Observe that by Corollary \ref{uniqb} we know that for
every cardinal $\lambda$ the number of non-isomorphic limit models is
bounded by $|\{\alpha : \alpha \leq |T|, \alpha \text{ is limit and }
\cof(\alpha)=\alpha \}| +1 $. So for example, when $R$ is countable,
we
know that there are at most two non-isomorphic limit models.
\end{remark}

We believe the following question is very interesting (see also Conjecture 2 of \cite{bovan}):

\begin{question}\label{quest}
Let $\K^{T}$ as in Hypothesis \ref{hyp}. How does the spectrum of
limit models look like? 

More precisely, given $\lambda$, how many non-isomorphic limit models
are there of cardinality $\lambda$ for a given $\K^T$? Is it always
possible to find $T$ such that $\K^T$ has the maximum number of
non-isomorphic limit models?
\end{question}

We will be able to answer Question \ref{quest} when the ring is
countable.

\begin{theorem} Let $R$ be a countable ring. Assume $\K^{T}$
satisfies Hypothesis \ref{hyp}.
\begin{enumerate}
\item If $\K^T$ is Galois-superstable\footnote{We say that $\K$ is
\emph{Galois-superstable} if there is $\mu < \beth_{(2^{\LS(\K)})^+}$
such that $\K$ is $\lambda$-Galois-stable for every $\lambda\geq \mu$. Under
the assumption of joint embedding, amalgamation, no maximal models
and $\LS(\K)$-tameness by \cite{grva} and \cite{vaseyt} the definition
of the previous line is equivalent to any other definition of
Galois-superstability given in the context of AECs.}, then there is $\mu <
\beth_{(2^{\aleph_0})^+}$  such for every $\lambda \geq \mu$ there is
a unique limit model of cardinality $\lambda$.
\item If $\K^T$ is not Galois-superstable, then $\K^T$ does not have
uniqueness of limit models in any infinite cardinal $\lambda\geq
\LS(\K^T)^+=\aleph_1$. More precisely, if $\K^T$ is $\lambda$-Galois-stable there
are exactly two non-isomorphic limit models of cardinality $\lambda$.
\end{enumerate}

\end{theorem}
\begin{proof}[Proof sketch] $\K^T$ has joint embedding, amalgamation and
no maximal models and by Corollary \ref{tame1} $\K^T$ is
$(<\aleph_0)$-tame. Due to this we can use the results of \cite{grva}
and \cite{vaseyt}.
\begin{enumerate} 
\item This follows on general grounds from \cite[4.24]{vaseyt} and
\cite[5.5]{grva}. 

\item Let $\lambda \geq \aleph_1$ such that $\K^T$ is
$\lambda$-Galois-stable. As in \cite[4.19, 4.20, 4.21, 4.23]{maz} one can
show that the limit models of countable cofinality are not
pure-injective. Since we know that limit models of uncountable
cofinality are pure-injective by Theorem \ref{bigpi}, we can conclude that
the $(\lambda,\omega)$-limit model and the $(\lambda,\omega_1)$-limit 
model are not isomorphic. Moreover, given $N$ a 
$(\lambda, \alpha)$-limit model, $N$ is isomorphic to the
$(\lambda,\omega)$-limit model if $cf(\alpha) =\omega$ 
(by a back-and-forth argument) or $N$ is isomorphic to the
$(\lambda,\omega_1)$-limit model if $cf(\alpha) >\omega$ 
(by Corollary \ref{uniqb}).
\end{enumerate}
\end{proof}

\subsection{Torsion-free abelian groups} In this section we will show
how to apply the results we just obtained to answer Question 4.25 of
\cite{maz}. 

Recall that a group $G$ is \emph{algebraically compact} if given
$\mathbb{E}=\{f_i(x_{i_0},...,x_{i_{n_i}})=a_i : i < \omega \}$ a set
of linear equations over $G$, $\mathbb{E}$ is finitely solvable in
$G$ if and only if  $\mathbb{E}$ is solvable in $G$. It is well-known that an abelian group $G$ is algebraically compact if and only if
$G$ is pure-injective (see e.g. \cite[1.2, 1.3]{fuchs}). The following theorem answers Question 4.25 of \cite{maz}.

\begin{theorem}\label{main1} If $G\in \K^{tf}$ is a $(\lambda, \omega)$-limit model, then $G \cong  \mathbb{Q}^{(\lambda)}
\oplus \prod_{p} \overline{\mathbb{Z}_{(p)}^{(\lambda)}}^{(\aleph_0)}$.

\end{theorem}
\begin{proof}
The amalgamation property together with the existence of a limit
model imply that $\K^{tf}$ is $\lambda$-Galois-stable. Then by Remark
\ref{st-2} $\lambda^{\aleph_0}=\lambda$, so by Corollary \ref{limo}
there is $H$ a $(\lambda, \omega_1)$-limit model.
Since $\K^{tf}$ is closed under direct sums, we have that $G \cong
H^{(\aleph_0)}$ by Theorem \ref{countablelim}.

In view of the fact that $H$ is a $(\lambda, \omega_1)$-limit model, 
by Fact
\ref{purec} $H \cong  \mathbb{Q}^{(\lambda)} \oplus \prod_{p}
\overline{\mathbb{Z}_{(p)}^{(\lambda)}}$. Therefore we have:

\[ G \cong \Big(\mathbb{Q}^{(\lambda)} \oplus \prod_{p}
\overline{\mathbb{Z}_{(p)}^{(\lambda)}} \Big)^{(\aleph_0)} \cong
\mathbb{Q}^{(\lambda)} \oplus \prod_{p}
\overline{\mathbb{Z}_{(p)}^{(\lambda)}}^{(\aleph_0)}. \] \end{proof}

In \cite[4.22]{maz} it was shown that limit models of countable
cofinality are not pure-injective. The argument given there uses some
deep facts about the theory of AECs. Here we give a new argument that
relies on some well-known properties of abelian groups.

\begin{cor}
If $G\in \K^{tf}$ is a $(\lambda, \omega)$-limit model,
then $G$ is not pure-injective.
\end{cor}
\begin{proof}
By Theorem \ref{main1} we have that $G  \cong  \mathbb{Q}^{(\lambda)}
\oplus \prod_{p} \overline{\mathbb{Z}_{(p)}^{(\lambda)}}^{(\aleph_0)}$.
For every $p$, one can show that
$\overline{\mathbb{Z}_{(p)}^{(\lambda)}}^{(\aleph_0)}$ is not
pure-injective by a similar argument to the proof that
$\overline{\mathbb{Z}_{(p)}}^{(\aleph_0)}$ is not pure-injective (an
argument for this is given in \cite[\S 2]{prest}). Then using that a
direct product is pure-injective if every component is pure-injective
(see \cite[\S 6.1.9]{fuchs}), it follows that $G$ is not
pure-injective. \end{proof}

Combining the results of this section with the ones of the previous
section we obtain:

\begin{cor} 
If $\forall \mu < \lambda( \mu^{\aleph_0} < \lambda)$, then for any
$G \in \K^{tf}_{\lambda}$ pure-injective there is a universal model
over it. 
\end{cor}
\begin{proof} 
Let $G \in \K^{tf}_{\lambda}$ be pure-injective. Since $\lambda$
satisfies the hypothesis of Corollary \ref{u-tf}, there is $U \in
\K^{tf}_{\lambda}$ universal model in $\K^{tf}_\lambda$. Then by
Lemma \ref{universal} $G \oplus U$ is a universal model over $G$.  \end{proof}

By the above corollary, given $G \in \K^{tf}_{\beth_\omega}$
pure-injective,  for example $G= \mathbb{Q}^{(\beth_\omega)}$,
there is $H \in \K^{tf}_{\beth_\omega}$ such that $H$ is universal
over $G$. Since $\beth_\omega^{\aleph_0} > \beth_\omega$, by Remark
\ref{st-2} we have that $\K^{tf}$ is not
$\beth_\omega$-Galois-stable. This is the first example of an AEC
with joint embedding, amalgamation and no maximal models in which one
can construct universal extensions of cardinality $\lambda$ without
the hypothesis of $\lambda$-Galois-stability.


\begin{thebibliography}{She01b}





\bibitem[Bal09]{baldwinbook09}
John Baldwin, \emph{Categoricity}, American Mathematical Society
(2009).


\bibitem[BCG+]{grp}
John Baldwin, Wesley Calvert, John Goodrick, Andres Villaveces and
Agatha Walczak-Typke, \emph{Abelian groups as aec's}. Preprint. URL: 
www.aimath.org/WWN/categoricity/abeliangroups\_10\_1\_3.tex.

\bibitem[BET07]{baldwine}
 John Baldwin, Paul Eklof and Jan Trlifaj, \emph{$^{\perp} N$ as an abstract
elementary class}, Annals of Pure and Applied Logic
\textbf{149}(2007), no. 1,25--39.

\bibitem[BoVan]{bovan}
Will Boney and Monica VanDieren, \emph{Limit Models in Strictly Stable Abstract Elementary Classes}, Preprint. URL: https://arxiv.org/abs/1508.04717 .

\bibitem[Bum65]{bumby}
Richard T. Bumby, \emph{Modules which are isomorphic to submodules of
each other}, Archiv der Mathematik \textbf{16}(1965), 184--185.

\bibitem[D\v{z}a05]{dz}
Mirna D\v{z}amonja, \emph{Club guessing and the universal models},
Notre Dame Journal of Formal Logic \textbf{46} (2005),  283--300. 

\bibitem[Ekl71]{eklof}
Paul Eklof, \emph{ Homogeneous universal modules}, Mathematica
Scandinavica \textbf{29}(1971), 187--196. 

\bibitem[Fuc15]{fuchs}
Laszlo Fuchs, \emph{Abelian group theory}, Springer (2015).

\bibitem[Gro02]{grossberg2002}
Rami Grossberg, \emph{Classification theory for abstract elementary
classes},
  Logic and Algebra (Yi~Zhang, ed.), vol. 302, American Mathematical
Society,
  2002, 165--204.

\bibitem[Gro1X]{ramibook}
Rami Grossberg, \emph{ A Course in Model Theory}, in Preparation,
201X.


\bibitem[GrSh83]{grsh}
Rami Grossberg and Saharon Shelah, \emph{On universal locally finite
groups}, Israel Journal of Mathematics \textbf{41} (1983), 289--302.

\bibitem[GrVan06]{tamenessone}
Rami Grossberg and Monica VanDieren, \emph{Galois-stability for tame
abstract elementary classes}, Journal
  of Mathematical Logic \textbf{6} (2006), no.~1, 25--49.

\bibitem[GrVas17]{grva}
Rami Grossberg and Sebastien Vasey, \emph{Equivalent definitions of
superstability in tame abstract elementary classes}, The Journal of
Symbolic Logic \textbf{82} (2017), no. 4, 1387 -- 1408.

\bibitem[GKS18]{gks}
Pedro A. Guil Asensio, Berke Kalebogaz and Ashish K. Srivastava,
\emph{The Schr\"{o}der-Bernstein problem for modules}, 
Journal of Algebra \textbf{498} (2018),
153--164.

\bibitem[KojSh95]{kojsh}
Menachem Kojman and Saharon Shelah, \emph{Universal abelian groups},
Israel Journal of Mathematics \textbf{92} (1995),  113--24.

\bibitem[KolSh96]{kosh}
Oren Kolman and Saharon Shelah, \emph{Categoricity of Theories in
$L_{\omega, \kappa}$ when $\kappa$ is a measurable cardinal. Part 1},
Fundamenta Mathematicae \textbf{151} (1996), 209--240.

\bibitem[MaSh76]{mash}
Angus Macintyre and Saharon Shelah, \emph{Uncountable universal
locally finite groups}, Journal Algebra \textbf{43} (1976), 168--175. 

\bibitem[Maz20]{maz}
Marcos Mazari-Armida, \emph{Algebraic description of limit models in classes of abelian groups}, Annals of Pure and Applied Logic \textbf{171} (2020), no. 1, 102723. 

\bibitem[MaVa18]{mv}
Marcos Mazari-Armida and Sebastien Vasey, \emph{Universal classes near $\aleph_1$}, The Journal of Symbolic Logic \textbf{83} (2018), no. 4, 1633–1643

\bibitem[Pre88]{prest}
Mike Prest, \emph{Model Theory and Modules}, London Mathematical
Society Lecture Notes Series Vol. 130, Cambridge University Press,
Cambridge (1988).

\bibitem[Pre09]{prest09}
Mike Prest, \emph{Purity, Spectra and Localisation (Encyclopedia of Mathematics and its Applications)},  Cambridge University Press, Cambridge (2009).

\bibitem[Sab70]{sab}
Gabriel Sabbagh, \emph{Aspects logique de la puret\'{e} dans les
modules}, C.R. Acad. Sci. Paris \textbf{271}(1970), 909 --912.


\bibitem[Sh87a]{sh88}
Saharon Shelah, \emph{Classification of nonelementary classes, {II}.
{A}bstract
  elementary classes}, Classification theory (John Baldwin, ed.)
(1987), 419--497.

\bibitem[Sh87b]{sh300}
Saharon Shelah, \emph{Universal classes}, Classification theory (John
Baldwin, ed.) (1987), 264--418.

\bibitem[Sh96]{sh1}
Saharon Shelah, \emph{Universal in $(< \lambda)$-stable abelian
group}, Mathematica Japonica \textbf{44}
(1996), 1--9.

\bibitem[Sh97]{sh2}
Saharon Shelah, \emph{Non-existence of universals for classes like
reduced torsion free abelian
groups under embeddings which are not necessarily pure}, 229--86 in
Advances in Algebra and Model Theory, edited by M. Droste and R.
Goebel, vol. 9 of Algebra, Logic and Applications, Gordon and Breach,
Amsterdam, 1997.

\bibitem[Sh01]{sh3}
Saharon Shelah, \emph{Non-existence of universal members in classes
of abelian groups}, Journal of Group Theory \textbf{4} (2001),
169--91.

\bibitem[Sh17]{sh820}
Saharon Shelah, \emph{Universal Structures}, Notre Dame Journal of
Formal Logic
    \textbf{58} (2017), no- 2 , 159--177.



\bibitem[Sh:h]{shelahaecbook}
Saharon Shelah, \emph{Classification Theory for Abstract Elementary Classes},
 vol. 1 \& 2, Mathematical Logic and Foundations, no. 18 \& 20, College
  Publications (2009).


\bibitem[ShVi99]{shvil}
Saharon Shelah and Andres Villaveces, \emph{Toward categoricity for
classes with no maximal models}, Annals of Pure and Applied Logic
\textbf{97} (1-3):1-25 (1999)

\bibitem[Vas16a]{vaseya}
Sebastien Vasey, \emph{Infinitary stability theory}, Archive for
Mathematical Logic \textbf{55} (2016), nos. 3-4, 562--592. 


\bibitem[Vas18]{vaseyt}
Sebastien Vasey, \emph{Toward a stability theory of tame abstract elementary classes}, Journal of Mathematical Logic \textbf{18} (2018), no. 2, 1850009.

  \bibitem[Zie84]{ziegler}
Martin Ziegler, \emph{Model Theory of Modules}, Annals of Pure and
Applied Logic \textbf{26}(1984), 149 -- 213.



\end{thebibliography}

\end{document}